\documentclass[reqno]{amsart}
\usepackage[utf8]{inputenc}
\usepackage[T1]{fontenc}
\usepackage[english]{babel}
\usepackage{mathtools}
\usepackage{amssymb}
\usepackage{amsthm}
\usepackage{MnSymbol}
\usepackage{bbm}
\usepackage{amssymb,fge}
\mathtoolsset{showonlyrefs,showmanualtags}

\newtheorem{Satz}{Satz}[section]
\newtheorem{Lem}[Satz]{Lemma}

\newtheorem{Cor}[Satz]{Corollary}
\newtheorem{Cla}[Satz]{Claim}

\newtheorem{Thm}[Satz]{Theorem}

\newcommand{\N}{\mathbb{N}}

\newcommand{\R}{\mathbb{R}}

\newcommand{\inv}{^{-1}}
\newcommand{\Leb}{\mathcal{L}}
\newcommand{\Ha}{\frac{1}{2}}
\newcommand{\Dc}{\bar{D}}
\newcommand{\Do}{D}
\newcommand{\K}{{S^1}}

\newcommand{\tn}[1]{\textnormal{#1}}

\newcommand{\tr}{\textrm{d}}
\makeatletter

\newcommand{\Rom}[1]{\expandafter\@slowromancap\romannumeral #1@}
\makeatother
\title{Space of minimal discs and its compactification}
\author{Paul Creutz}
\address{Paul Creutz, Mathematisches Institut der Universit\"at zu K\"oln, Weyertal 86-90, 50931 K\"oln, Germany}
\email{pcreutz@math.uni-koeln.de}
\thanks{The author was partially supported by the DFG grant SPP 2026.}
\begin{document}
\begin{abstract}
We investigate the class of geodesic metric discs satisfying a uniform quadratic isoperimetric inequality and uniform bounds on the length of the boundary circle. We show that the closure of this class as a subset of Gromov-Hausdorff space is intimately related to the class of geodesic metric disc retracts satisfying comparable bounds. This kind of discs naturally come up in the context of the solution of Plateau's problem in metric spaces by Lytchak and Wenger as generalizations of minimal surfaces.
\end{abstract}
\maketitle
\section{Introduction}
\label{1}
\subsection{Main result}
\label{1.1}
Since Gromov stated his precompactness criterion in \cite{Gro81a} the study of compact and precompact subsets of Gromov-Hausdorff space of metric spaces became a vivid field. Given a precompact class of metric spaces it is a natural but usually very hard question to determine its closure with respect to Gromov-Hausdorff distance. Various partial results have been obtained in the situation where one considers classes of Riemannian manifolds satisfying uniform bounds on parameters such as curvature, volume and diameter. See for example \cite{Gro99}, \cite{GPW90}, \cite{BGP92} and \cite{CC97} for some of the most important ones. 
In this article we investigate the class of geodesic metric discs satisfying a uniform quadratic isoperimetric inequality and upper bound on the length of the boundary circle. Such discs come up in \cite{LW18a} as a replacement for minimal surfaces in a quite general metric space setting. The main result of this paper shows that the closure of this class is intimately related to the space of geodesic metric disc retracts in the sense of \cite{PS18} satisfying comparable bounds. One should note that even without further geometric assumptions it should be possible to understand the topology of Gromov-Hausdorff limits of geodesic metric discs. Compare for example the $2$-dimensional results in \cite{Why35}, \cite{Shi99} and \cite{Gro99}. However the arising spaces will be topologically more complicated than disc retracts in general.\par
We say that a geodesic metric disc $Z$ satisfies a \textit{$C$-quadratic isoperimetric inequality} if for every Jordan domain $U\subset Z$ one has
\begin{equation}
\label{e1}
\mathcal{H}^2(U)\leq C \cdot l(\partial U)^2.
\end{equation}
We call a metric space $Z$ a \textit{disc retract} if there exists a closed curve $\gamma:\K\rightarrow X$ such that the mapping cylinder~$Z_\gamma$ is a topological disc. We write $l(\partial Z)\leq L$ if there exists such~$\gamma$ satisfying~$l(\gamma)\leq L$.
\begin{Thm}
\label{t1}
Let $C,L\in (0,\infty)$. Denote by $\mathcal{D}(L,C)$ the set of geodesic metric discs $Z$ satisfying a $C$-quadratic isoperimetric inequality and $l(\partial Z)\leq L$. And by $\mathcal{E}(L,C)$ the set of geodesic metric disc retracts satisfying the same two properties.\par 
Then $\mathcal{E}(L,C)$ is compact and
\begin{equation}
\overline{\mathcal{D}(L,C)}\subset \mathcal{E}(L,C)\subset \overline{\mathcal{D}(L,C+(2\pi)\inv)}.
\end{equation}
where $\overline{\mathcal{D}(L,C)}$ denotes the closure of $\mathcal{D}(L,C)$ with respect to Gromov-Hausdorff distance.
\end{Thm}
It has been shown in \cite{LW18b} that a proper geodesic metric space $X$ satisfies a $(4\pi)\inv$-quadratic isoperimetric inequality iff it is a $\tn{CAT}(0)$-space. In this light our theorem~\ref{t1} covers as a special case the compactness lemma in \cite{PS18} which states that the class of $\tn{CAT}(0)$ disc retracts $Z$ satisfying $l(\partial Z)\leq L$ is compact.\par 
Note that by \cite{LW17a} for a sufficiently small constant $C$ every compact geodesic metric space satisfying a $C$-quadratic isoperimetric inequality is a metric tree and hence not a metric disc. So at least in general one cannot have $\overline{\mathcal{D}(L,C)}= \mathcal{E}(L,C)$.
\subsection{Byproducts}
\label{1.2}
Classical problem of Plateau formulated in the metric space setting asks whether for a given Jordan curve $\Gamma$ in a metric space~$X$ there exists a Sobolev disc $u \in W^{1,2}(\Do,X)$ of least area spanning~$\Gamma$. Lytchak and Wenger solved Plateau's problem for proper metric spaces in \cite{LW17a}. Similar to the intrinsic metric on an immersed submanifold to a solution $u$ of Plateau's problem for a given Jordan curve $\Gamma$ one may associate a geodesic metric space $Z_u$.  If $X$ satisfies a $C$-quadratic isoperimetric inequality, then $Z_u$ is a geodesic metric disc satisfying a $C$-quadratic isoperimetric inequality and $l(\partial Z_u)=l(\Gamma)$, see~\cite{LW18a}. Here we say that $X$ satisfies a \textit{$C$-quadratic isoperimetric inequality} if every closed Lipschitz curve $\gamma:\K \rightarrow X$ admits a disc Sobolev disc $v\in W^{1,2}(\Do,X)$ spanning $\gamma$ such that 
\[
\tn{Area}(v)\leq C \cdot l(\gamma)^2.
\]
In section~\ref{2.3} we prove the following.
\begin{Thm}
\label{t2}
Let $Z\in \mathcal{D}(L,C)$ and $u\in W^{1,2}(\Do,Z)$ be a solution of Plateau's problem for the curve $\partial Z$ in $Z$. Then $Z_u$ isometric to $Z$.
\end{Thm}
So it is natural to think of elements of $\mathcal{D}(L,C)$ as intrinsic minimal discs. Especially theorem~\ref{t2} allows to apply the structural results of \cite{LW18a} to elements of~$\mathcal{D}_{L,C}$. One consequence of this being that~$\mathcal{D}_{L,C}$ is precompact in Gromov-Hausdorff space.\par 
If additionally one imposes a uniform condition avoiding degeneration of the boundary curve, then one only needs to adjoin the trivial metric space~$\{*\}$ to compactify. A Jordan curve $\Gamma$ in a metric space $ X$ is said to be $\lambda$-\textit{chord-arc} where $\lambda \geq 1$ if for all $x,y \in \Gamma$ the length of the shorter of the two arcs connecting $x$ and $y$ in $\Gamma$ is bounded above by~$\lambda \cdot d(x,y)$.
\begin{Thm}
\label{t3}
Let $C,L \in (0,\infty)$, $\lambda \in [1,\infty)$ and $\mathcal{D}^\lambda(L,C)$ be the set of those $Z\in \mathcal{D}(L,C)$ such that $\partial Z$ is $\lambda$-chord-arc. Then $\mathcal{D}^\lambda(L,C)\cup \{*\}$ is compact.
\end{Thm}
The proof is based on a result from~\cite{Why35} which shows that under certain uniform local contractibility conditions Hausdorff limits of discs are again discs.\par
In section~\ref{3} we discuss how to associate to a disc retract $Z\in \mathcal{E}(L,C)$ a metric on the mapping cylinder $Z_\gamma$ such that $Z_\gamma\in \mathcal{D}^1(L,C+1)$. This general construction might be of independent interest and is also applied in \cite{Crearb}. Using this observation it is not to hard to derive theorem~\ref{t1} from theorem \ref{t3}.\par
We call a monotone increasing function $\rho:(0,R_0)\rightarrow \R$ a \textit{contractibilty function} if $\lim_{r\rightarrow 0}\rho(r)=0$. We say that a metric space $X$ is in $\tn{LGC}(\rho)$ if for every $x\in X$ and $r\in (0,R_0)$ the ball $B(x,r)$ is contractible inside $B(x,\rho(r))$. Using Whyburn's result in \cite{Why35} instead of theorem~\ref{t3} and the same construction we also obtain.
\begin{Thm}
Let $\rho$ be a contractibility function, $L \in (0,\infty)$ and $(Z_n)$ be a sequence of disc retracts such that $Z_n\in \tn{LGC}(\rho)$ and $l(\partial Z_n)\leq L$. If $(Z_n)$ GH-converges to a compact metric space~$Z$, then $Z$ is a disc retract and $l(\partial Z)\leq L$.
\end{Thm}
\subsection{Local versions}
\label{1.3}
Let $C\in (0,\infty)$ and $l_0\in(0,\infty]$. A metric space $X$ is said to satisfy a $(C,l_0)$-\textit{(local) quadratic isoperimetric inequality} if for every closed Lipschitz curve $\gamma:\K \rightarrow X$ such that $l(\gamma)<l_0$ there exists a disc $u\in W^{1,2}(\Do,X)$ such that $\tn{tr}(u)=\gamma$ and $\tn{Area}(u)\leq C \cdot l(\gamma)$. Taking $l_0=\infty$ this definition is compatible with the definition given in~\ref{1.1}, see \cite{Crearb}. Easy examples show that the class of geodesic metric discs $Z$ satisfying a $(C,l_0)$-quadratic isoperimetric inequality and $l(\partial Z)\leq L$ is not precompact in Gromov-Hausdorff space. Think for example of cylinders $\K \times [0,R]$ attached a spherical cap at one end. However this may be fixed by bounding also the two dimensional Hausdorff measure of $Z$.\par 
Let $\mathcal{D}(L,A,C,l_0)$ be the class of geodesic metric discs $Z$ satisfying a $(C,l_0)$-quadratic isoperimetric inequality, $l(\partial Z)\leq L$ and $\mathcal{H}^2(Z)\leq A$. Let $\mathcal{E}(L,A,C,l_0)$ be the class of geodesic metric disc retracts satisfying the same bounds. Then we prove that $\mathcal{D}(L,A,C,l_0)$ is precompact and adequately adjusted variants of theorem~\ref{t1}, theorem~\ref{t2} and theorem~\ref{t3} hold for this class replacing $\mathcal{E}(L,C)$ by $\mathcal{E}(L,A,C,l_0)$ and alike. We will only prove the more general versions below which cover the ones stated above for the special case $l_0=\infty$ and $A=C\cdot L^2$. 
\section{Plateau's problem and intrinsic minimal discs}
\label{2}
\subsection{Reminder: Metric space valued Sobolev maps}
\label{2.1}
We give a short reminder on metric space valued Sobolev maps. For more details see for example \cite{LW17a}, \cite{KS93} and \cite{Res97}.\par
Let $\Do:=\{v\in \R^2 | \ |v|\leq 1\}$ and $\K:=\{v \in \R^2| \ |v|=1 \}$, $X$ a complete separable metric space and $p\in(1,\infty)$. A measurable map $u: \Do \rightarrow X$ belongs to $L^p(\Do,X)$ if for every $1$-Lipschitz map $f:X\rightarrow \R$ the composition $f\circ u$ belongs to $L^p(\Do)$. A map $u\in L^p(\Do,X)$ belongs to $W^{1,p}(\Do,X)$ if for every $1$-Lipschitz map $f:X\rightarrow \R$ the composition $f\circ u$ lies in the classical Sobolev space $W^{1,p}(\Do)$ and the \textit{Reshtnyak p-energy} of $u$
\begin{equation}
\label{e2}
E^p_+(u):=\inf \left\{||g||^p_{L^p(\Do)}\ \Big|\ g \in L^p(\Do):\forall f\in \tn{Lip}_1(X,\R): |\nabla (f\circ u)|\leq g\tn{ a.e.}\right\}
\end{equation}
is finite. \par
We say that a map $u:\Do\rightarrow X$ satisfies \textit{(Lusin's) property (N)} if for every $N\subset \Do$ with $\Leb^2(N)=0$ one has $\mathcal{H}^2(u(N))=0$. For $p>2$ the following Sobolev embedding theorem holds, see for example~\cite{LW17a}.
\begin{Thm}
\label{t4}
Let $p>2$ and $\alpha:=1-\frac{2}{p}$. If $u \in W^{1,p}(\Do,X)$, then $u$ has a representative $\bar{u} \in C^\alpha(\Dc,X)$ satisfying Lusin's property (N). Furthermore there is a constant $M=M(p)$ such that
\begin{equation}
\label{e3}
||u||_\alpha:=\sup_{x,y \in \Dc} \frac{d(u(x),u(y))}{|x-y|^\alpha}\leq M\cdot \left(E^p_+(u)\right)^{\frac{1}{p}}.
\end{equation}
\end{Thm}
If $u\in W^{1,2}(\Do,X)$, then $u$ is approximately metrically differentiable almost everywhere. That is for almost every $x \in \Do$ there exists a seminorm $\tn{apmd}_x u$ on $\R^2$ such that
\begin{equation}
\label{e4}
\tn{aplim}_{y\rightarrow x} \frac{d(u(y),u(x))-\left(\tn{apmd}_x u\right)(y-x)}{|y-x|}=0.
\end{equation}
Define the \textit{(Busemann) area} $\tn{Area}(u)\in [0, \infty)$ of $u \in W^{1,2}(\Do,X)$ by
\begin{equation}
\label{e5}
\tn{Area}(u):=\int_\Do J(\tn{apmd}_x u)\ \tr \Leb^2(x)\leq E^2_+(u)
\end{equation}
where for a seminorm $\sigma$ on $\R^2$ one sets its \textit{(Busemann) Jacobian} $J(\sigma)$ to be
\begin{equation}
\label{e6}
J(\sigma):=\frac{\pi}{\Leb^2(\{v\in \R^2|\sigma(v)\leq 1\})}.
\end{equation}
It follows from the definition that for a Möbius transform $m:\Do\rightarrow \Do$ one has $\tn{Area}(u\circ m)=\tn{Area}(u)$ and $E^2_+(u\circ m)=E^2_+(u)$. We have the following variant of the area formula which especially shows, that our definition extends the one used in Riemannian geometry.
\begin{Thm}[\cite{Kar07}]
\label{t5}
If $u\in W^{1,2}(\Do,X)$ satisfies Lusin's property (N), then
\begin{equation}
\label{e7}
\tn{Area}(u)=\int_X \tn{card}( u\inv (x)) \ \tr  \mathcal{H}^2(x).
\end{equation}
\end{Thm}
For $u\in W^{1,2}(\Do,X)$ there is a canonical almost everywhere defined \textit{trace} map $\tn{tr}(u)\in L^2(\K, X)$. We write $u\in C^0(\Dc,X)$ if $u$ extends to a continuous map $\bar{u}:\Dc\rightarrow X$. In this case one may simply take $\tn{tr}(u)=u_{|\K}$.\par  
Let $I=[a,b]$ or $I=\K$. Let $\alpha,\beta:I\rightarrow X$ and $h:I\times [0,1] \rightarrow X$ be Lipschitz maps. We say that $h$ is a Lipschitz homotopy from $\alpha$ to $\beta$ if $h(-,0)=\alpha$ and $h(-,1)=\beta$. In this case we write $h:\alpha \rightarrow \beta$. The area of such a Lipschitz homotopy $h$ may be defined similarly as discussed here for Sobolev maps on the disc, see \cite{Kir94}. The following glueing lemma is very useful and will be used repeatedly, see for example section $2.2$ in \cite{LW16}.
\begin{Lem}
\label{l1}
Let $\gamma,\eta:\K\rightarrow X$ be closed Lipschitz curves and $h:\gamma \rightarrow \eta$ a Lipschitz homotopy. For $u \in W^{1,2}(\Do,X)$ such that $\tn{tr}(u)=\gamma$, there exists $v\in W^{1,2}(\Do,X)$ such that $\tn{tr}(v)=\eta$ and
\begin{equation}
\tn{Area}(v)=\tn{Area}(h)+\tn{Area}(u).
\end{equation}
\end{Lem}
Also the following lower semicontinuouity result will be needed at some point.
\begin{Thm}[\cite{LW17a}, \cite{Res97}]
\label{t6}
Let $p\geq 2$, $E \geq 0$ and $(u_j)\subset W^{1,p}(\Do,X)$. If $E^p_+(u_j)\leq E$ for all $j$ and $u_j\rightarrow u\in L^p(\Do,X)$, then $u \in W^{1,p}(\Do,X)$ and 
\begin{equation}
\label{e8}
\tn{Area}(u)\leq \liminf \tn{Area}(u_j).
\end{equation} 
\end{Thm}
\subsection{Plateau's problem in metric spaces}
Call a continuous map $f:X\rightarrow Y$ between metric spaces $X,Y$ \textit{monotone} if $f$ is continuous, surjective and $f\inv(y)$ is connected for every $y \in Y$.\par
Let $X$ be a complete metric space and $\Gamma$ a closed Jordan curve in $X$. Let $\Lambda(\Gamma,X)$ be the set of those $u \in W^{1,2}(\Do,X)$ such that $\tn{tr}(u)\in L^2(\K,X)$ has a representative which is a monotone map $\K\rightarrow \Gamma$. Following \cite{LW18a} we call $u \in \Lambda(\Gamma,X)$ a \textit{solution of Plateau's problem} for $\Gamma$ in $X$ if
\begin{equation}
\label{e9}
\tn{Area}(u)=\inf_{v\in \Lambda(\Gamma,X)} \tn{Area}(v)
\end{equation} 
and $u$ minimizes $E^2_+(u)$ among all $u \in \Lambda(\Gamma,X)$ satisfying \eqref{e9}.
\begin{Thm}[\cite{LW17a}]
\label{t7}
Let $X$ be a proper metric space and $\Gamma$ a closed Jordan curve in $X$. If $\Lambda(\Gamma,X)\neq \emptyset$, then there exists a solution~$u$ of Plateau's problem for~$\Gamma$ in~$X$. Every such $u$ satisfies
\begin{equation}
\label{e10}
E^2_+(u)\leq 2\cdot \tn{Area}(u).
\end{equation}
\end{Thm}
Here a metric space $X$ is called \textit{proper} if closed and bounded subsets of $X$ are compact. To gain higher regularity of $u$ one has to impose additional conditions on $X$ and $\Gamma$.\par 
Let $C \in (0,\infty)$ and $l_0\in (0,\infty]$. We say that $X$ satisfies a \textit{$(C,l_0)$-quadratic isoperimetric inequality} if for every Lipschitz curve $\gamma:\K \rightarrow X$ satisfying $l(\gamma)<l_0$ there exists $u\in W^{1,2}(\Do,X)$ such that $\tn{tr}(u)=\gamma$ and 
\begin{equation}
\label{e11}
\tn{Area}(u)\leq C \cdot l(\gamma)^2.
\end{equation}
If $l_0=\infty$, we also simply say that $X$ satisfies a $C$-quadratic isoperimetric inequality. I turns out that a geodesic metric disc $Z$ satisfies a $C$-quadratic isoperimetric inequality iff every Jordan domain $U\subset Z$ satisfies $\mathcal{H}^2(U)\leq C \cdot l(\partial U)^2$, see~\cite{Crearb},~\cite{LW17b}. For example compact Riemannian and more generally Finsler manifolds, $\tn{CAT}(\kappa)$-spaces and Banach spaces satisfy some quadratic isoperimetric inequality, see \cite{LW17a}, \cite{Creara}.
\begin{Thm}[\cite{LW17a}, \cite{LW16}]
\label{t8}
Let $X$ be a complete metric space satisfying a $(C,l_0)$-quadratic isoperimetric inequality, $\Gamma$ a closed rectifiable Jordan curve in $X$ and $u$ a solution of Plateau's problem for~$\Gamma$ in~$X$.\par
Then $u$ has a representative $\bar{u}$ such that $\bar{u} \in C^0(\Dc,X)$ and $\bar{u}$ satisfies Lusin's property~(N). If furthermore $\Gamma$ is $\lambda$-chord-arc and $l(\Gamma)<l_0$, then one may find $p=p(C,\lambda)>2$ such that $\bar{u} \in W^{1,p}(\Do,X)$  and a Möbius transform $m:\Do\rightarrow \Do$ such that
\begin{equation}
\label{e12}
\left(E^p_+(\bar{u}\circ m)\right)^{\frac{1}{p}}\leq M\cdot \left(E^2_+(u)\right)^{\frac{1}{2}}
\end{equation}
where $M=M(C,\lambda)$.
\end{Thm}
So if the target spaces satisfies some quadratic isoperimetric inequality we may always without loss of generality assume that a solution of Plateau's problem $u$ is in $C^0(\Dc,X)$ and satisfies Lusin's property (N). For disc targets one obtains the following improvement.
\begin{Thm}[\cite{LW17b}]
\label{t9}
Let $Z\in \mathcal{D}(L,A,C,l_0)$. Then there exists a solution $u$ of Plateau's problem for $\partial Z$ in $Z$. Every such $u$ is monotone and furthermore satisfies~$\tn{Area}(u)=\mathcal{H}^2(Z)$.
\end{Thm}
This formulation of theorem~\ref{t9} is contained only implicitly in~\cite{LW17b}. So for the convenience of the reader we reproduce the proof.
\begin{proof}
Every Jordan domain $U\subset Z$ satisfies
\begin{equation}
\label{e13}
\mathcal{H}^2(U)\leq C'\cdot l(\partial U)^2
\end{equation} 
where $C'=\max\left\{C,\frac{\mathcal{H}^2(Z)}{l_0^2}\right\}$. So $Z$ satisfies a $C'$-quadratic isoperimetric inequality. As $l(\partial Z)<\infty$ the boundary curve admits a Lipschitz parametrization and hence by the global quadratic isoperimetric inequality one has $\Lambda(\partial Z,Z)\neq \emptyset$. So by theorem~\ref{t7} there is a solution of Plateau's problem for~$\partial Z$ in~$Z$.\par 
We prove that $u$ is a solution of Plateau's problem for $\partial Z$ in $Z$ iff 
\begin{equation}
\label{e14}
E^2_+(u)=\inf_{w\in\Lambda(\partial Z,Z)}E^2_+(w).
\end{equation}
Let $u$ be a solution of Plateau's problem for~$\partial Z$ in~$Z$ and $v \in \Lambda(\partial Z,Z)$ be such that $E^2_+(v)=\inf_{w \in \Lambda(\partial Z,Z)} E^2_+(w)$. By \cite[theorem 4.4]{LW17c} such $v$ exists and may be chosen in $C^0(\Dc,X)$ and satisfying Lusin's property (N). So by theorem~\ref{t5}
\begin{equation}
\label{e15}
\int_Z \tn{card}(v\inv (z))\ \tr \mathcal{H}^2(z)=\tn{Area}(v)\leq E^2_+(v)<\infty.
\end{equation}
So $\tn{card}(v\inv (z))$ is finite~$\mathcal{H}^2$ almost everywhere in~$Z$. But by \cite[theorem~1.2]{LW17b} $v$ is monotone and hence $\tn{card}(v\inv(z))=1$ for $\mathcal{H}^2$-almost every $z\in Z$. Using that $u$ satisfies Lusin's property (N) and applying theorem~\ref{t5} again we obtain
\begin{equation}
\label{e16}
\tn{Area}(u)\leq \tn{Area}(v)=\mathcal{H}^2(Z)\leq \tn{Area}(u).
\end{equation}
So $v$ is a solution of Plateau's problem. As $u$ minimizes $E^2_+$ among all area minimizers vice versa one has that 
\begin{equation}
\label{e17}
E^2_+(u)\leq E^2_+(v)=\inf_{w \in \Lambda(\partial Z,Z)} E^2_+(w)\leq E^2_+(u).
\end{equation} 
As minimizers of $E^2_+$ are monotone and have area equal to $\mathcal{H}^2(Z)$ this completes the proof.
\end{proof}
\subsection{Intrinsic minimal discs}
\label{2.3}
Let $X$ be a complete metric space and $u \in C^0(\Dc,X)$. Define a semimetric $d_u:\Dc\times \Dc \rightarrow [0,\infty]$ on $\Dc$ by setting
\begin{equation}
\label{e18}
d_u(x,y):=\inf\{ l(u \circ \eta)\ | \ \eta:[0,1]\rightarrow \Dc \tn{ continuous, }\eta(0)=x,\eta(1)=y\}.
\end{equation}
We call the metric space $Z_u$ corresponding to $(\Dc,d_u)$ the \textit{intrinsic disc} associated to $u$. By definition one gets canonical maps $P_u:\Dc\rightarrow Z_u$ and $\hat{u}:Z_u \rightarrow X$ such that $\hat{u}\circ P_u=u$. For solutions of Plateau's problem $u$ this construction has been studied in detail in \cite{LW18a}.
\begin{Thm}[\cite{LW18a}]
\label{t10}
Let $X$ be a complete metric space satisfying a $(C,l_0)$-quadratic isoperimetric inequality, $\Gamma$ a rectifiable Jordan curve in $X$ and $u$ a solution of Plateau's problem for $\Gamma$ in $X$. Then we have the following list of properties.
\begin{enumerate}
\item \label{en1}
$Z_u$ is a geodesic metric disc satisfying $\mathcal{H}^2(U)\leq C \cdot l(\partial U)^2$ for every Jordan domain $U \subset Z_u$ such that $l(\partial U)<l_0$.
\item \label{en2}
$P_u\in \Lambda(\partial Z_u,Z_u)$ is a solution of Plateau's problem for $\partial Z_u$ in $Z_u$.
\item \label{en3}
For every curve $\gamma$ in $\Dc$ one has $l(u \circ \gamma)=l(P_u\circ \gamma)$ and for every $V \subset \Dc$ open $\tn{Area}(u_{|V})=\mathcal{H}^2(P_u(V))$. Especially $l(\Gamma)=l(\partial Z_u)$ and $\tn{Area}(u)=\mathcal{H}^2(Z_u)$.
\item \label{en4}
$\hat{u}$ is $1$-Lipschitz.
\item \label{en5}
$Z_u \in \tn{LGC}(\rho)$ where $\rho:\left(0,\frac{l_0}{2}\right)\rightarrow\R$, $\rho(r):=(8C+1)\cdot r$.
\item \label{en6}
Let $l(\Gamma)\leq L$ and $\tn{Area}(u)\leq A$. Then for every $\epsilon>0$ there exists an $\epsilon$-net in $Z_u$ with at most $K$ elements where $K=K(\epsilon,L,A,C,l_0)$.
\end{enumerate}
\end{Thm}
We prove the following more precise version of theorem~\ref{t2}.
\begin{Thm}
\label{t11}
Let $Z\in \mathcal{D}(L,A,C,l_0)$. If $u$ is a solution of Plateau's problem for $\partial Z$ in $Z$, then $\hat{u}:Z_u\rightarrow Z$ is an isometry.
\end{Thm}
For $u \in C^0(\Dc,X)$ two variants of the intrinsic disc metric have been investigated in \cite{PS18}. Define a semimetric $|d_u|:\Dc\times \Dc \rightarrow \R$ by setting
\begin{equation}
\label{e19}
|d_u|(x,y):=\inf\{\tn{diam}(u(C))\ | \ C\subset \Dc \tn{ connected, } x,y\in C\}.
\end{equation}
Let $|Z_u|$ be the metric space corresponding to $(\Dc,|d_u|)$. Let $\langle|Z_u|\rangle$ be the metric space obtained by taking the path metric on $|Z_u|$. As $Z_u$ is a length space we get the following canonical factorization of $\hat{u}$ by $1$-Lipschitz maps 
\begin{equation}
\label{e20}
Z_u\overset{u_1}{\rightarrow}\langle|Z_u|\rangle \overset{u_2}{\rightarrow}|Z_u| \overset{u_3}{\rightarrow}X.
\end{equation}
\begin{proof}
The aim is to show that in our situation all the maps $u_i$ in \eqref{e20} are isometries and hence so is $\hat{u}$.\par 
\textbf{$u_3$ is an isometry:} Let $x,y \in \Dc$. Let $\gamma:[0,1]\rightarrow X$ be a geodesic in $Z$ connecting $u(x)$ and $u(y)$. As $\Dc$ is compact and $u$ is monotone the set $u\inv(\tn{im}(\gamma))$ is connected. So
\[
|d_u|(x,y)\leq \tn{diam}(u(u\inv(\tn{im}(\gamma))))=\tn{diam}(\tn{im}(\gamma))=d(u(x),u(y)).
\]
As $u_3$ is $1$-Lipschitz this proves that $u_3$ is an isometry.\par 
\textbf{$u_2$ is an isometry:} As $|Z_u|$ is isometric to $Z$ also $|Z_u|$ is geodesic. So $u_2$ is an isometry.\par
\textbf{$u_1$ is an isometry:} We say that $u$ has \textit{no bubbles} if for every $z \in Z$ every connected component of $\Dc \setminus u\inv(z)$ contains a point in $\K$. In \cite[proposition~9.3]{PS18} it has been proved that if $P_u:\Dc\rightarrow Z_u$ is continuous and $u$ has no bubbles, then $u_1:Z_u\rightarrow \langle | Z_u|\rangle$ is an isometry. By theorem~\ref{t10}.\ref{en2} $P_u$ is continuous. So it suffices to proof that $u$ has no bubbles. A proof of this fact is implicitly contained in \cite{LW18a}. In our situation one can also conclude this by the following simple argument.\par 
Assume $U$ was a connected component of $\Dc \setminus u \inv(z)$ disjoint to $\K$. Then $U$ is open. Define $\tilde{u}:\Dc\rightarrow Z$ by setting $\tilde{u}(q):=z$ for $q \in U$ and $\tilde{u}(q):=u(q)$ otherwise. Then $\tilde{u}$ is continuous and restricts to a monotone parametrization of $\partial Z$ on $\K$. So by topological reasons $\tilde{u}$ is surjective. Let $q \in U$ and $w:=u(q)$. Then there is $p \in \Dc \setminus \overline{U}$ such that $u(p)=\tilde{u}(p)=w$. As $\partial U \cap u\inv(w)=\emptyset$ this proves that $u\inv(w)$ is disconnected. A contradiction to $u$ being monotone.
\end{proof}
\begin{Cor}
\label{c1}
For $L,A,C \in (0,\infty)$ and $l_0\in (0,\infty]$ let $\mathcal{D}(L,A,C,l_0)$ be the class of all geodesic metric discs $Z$ satisfying $l(\partial Z)\leq L$, $\mathcal{H}^2(Z)\leq A$ and a $(C,l_0)$-quadratic isoperimetric inequality. Then $\mathcal{D}(L,A,C,l_0)$ is precompact in Gromov-Hausdorff space.
\end{Cor}
\begin{proof}
Let $\epsilon >0$. By theorem~\ref{t11} and theorem~\ref{t10}.\ref{en6} there is a universal constant~$K$ such that every $Z \in\mathcal{D}(L,A,C,l_0)$ admits an $\epsilon$-net with at most $K$ elements. So by Gromov's compactness criterion, see \cite{Gro81a}, $\mathcal{D}(L,A,C,l_0)$ is precompact in Gromov-Hausdorff space. Note here that the uniform diameter bound needed in the criterion is automatic for classes consisting of geodesic spaces, see for example~\cite[exercise 7.4.14]{BBI01}.
\end{proof}
\section{Mapping cylinders}
\label{3}
Let $X$ be a compact topological space, $\gamma:\K \rightarrow X$ a closed continuous curve and $I=[0,1]$. The \textit{mapping cylinder} $X_\gamma$ of $\gamma$ is the topological space obtained by glueing $Y:=\K \times I$ to $X$ along the equivalence relation generated by $(t,0)\sim\gamma(t)$.\par
Now assume $X$ is a metric space, $\gamma:\K \rightarrow X$ is Lipschitz and $L,R\in (0,\infty)$. Let $d_Y$ be the product metric on $Y$ obtained by scaling the standard metric on $I$ by factor~$R$ and the angular distance on $Y$ by factor~$\frac{L}{2\pi}$. Let $d$ be the maximal semimetric on $X_\gamma$ subject to $d(x,y)\leq d_X(x,y)$ for $x,y \in X$ and $d(x,y)\leq d_Y(x,y)$ for $x,y \in Y$.
\begin{Lem}
\label{l2}
If  $L \geq 2\pi \cdot \tn{Lip}(\gamma)$, then $d$ is a metric and metrizes the topology of~$X_\gamma$. We denote $X_{\gamma,L,R}:=(X_\gamma,d)$.
\end{Lem}
In lemma~\ref{l2} we measure $\tn{Lip}(\gamma)$ with respect to angular distance on $\K$.
\begin{proof}
For $x,y \in X_\gamma$ set
\begin{equation}
\label{e21}
d'(x,y):=\begin{cases}
d_X(x,y) &;x,y \in X\\
\underset{t\in \K}{\min} \left\{d_X(x,\gamma(t))+d_Y((t,0),y) \right\}&; x \in X, y\in Y\\
\scriptstyle \underset{t,s\in \K}{\min}\left\{d_{Y}(x,y),d_Y(x,(t,0))+d_X(\gamma(t),\gamma(s))+d_Y((s,0),y)\right\} &; x,y \in Y.
\end{cases}
\end{equation}
We claim that $d'$ is a metric. It is clear that $d'$ is well defined, positive definite and symmetric. It remains to check the triangle inequality. Let $x,y,z \in X_\gamma$. We treat the case $x,z\in X$ and $y\in Y$ only. All other ones are similar. There exist $t,s \in \K$ such that
\begin{align}
d'(x,y)+d'(y,z)&=d_X(x,\gamma(t))+d_Y((t,0),y)+d_Y(y,(s,0))+d_X(\gamma(s),z)\\
&\geq d_X(x,\gamma(p))+\cdot d_X(\gamma(t),\gamma(s))+d_X(\gamma(s),z)\\
&\geq d_X(x,z)=d'(x,z).
\end{align}
So $d'$ is a metric and hence $d'\leq d$. But due to the triangle inequality also $d\geq d'$ and hence $d=d'$. As $d$ is a metric and $X\sqcup Y$ is compact the topology generated by $d$ agrees with the quotient topology on $Z$, see for example exercise~$3.1.14$ in~\cite{BBI01}.
\end{proof}
\begin{Thm}
\label{t12}
In the situation of lemma~\ref{l2} we have the following list of geometric properties.
\begin{enumerate}
\item \label{en7}
The inlcusion $X\rightarrow X_{\gamma,L,R}$ is an isometric embedding and
$X$ is an $R$-net in~$X_{\gamma,L,R}$. There is a strong deformation retraction $H$ of $X_{\gamma,L,R}$ onto $X$ by $1$-Lipschitz maps $H_t:X_{\gamma,L,R}\rightarrow X_{\gamma,L,R}$.
\item \label{en8}
If $X$ is geodesic, then $X_{\gamma,L,R}$ is geodesic.
\item \label{en9}
$\mathcal{H}^2(X_{\gamma,L,R})=\mathcal{H}^2(X)+LR$.
\item \label{en10}
If $R\geq L$, then the curve $\gamma_R:\K \rightarrow X_{\gamma,L,R},\gamma_R(p):=(p,1)$ is a $1$-chord-arc curve of length $L$.
\item \label{en11}
If $X$ is $\tn{LGC}(\rho)$ where $\rho:(0,R_0)\rightarrow\R$ is a contractibility function, then $X_{\gamma,L,R}$ is $\tn{LGC}(\bar{\rho})$ where $R_0':=\min\left\{\frac{R_0}{2},\frac{L}{2}\right\}$, $\bar{\rho}:\left(0,R_0'\right)\rightarrow \R, \bar{\rho}(r)=\rho(2r)+r$.
\item \label{en12}
If $X$ satisfies a $(C,l_0)$-quadratic isoperimetric inequality, then $X_{\gamma,L,R}$ satisfies a $(C',l_0)$-quadratic isoperimetric inequality where $C'=C+\max\left\{ \frac{1}{2\pi}, \frac{R}{L}\right\}$.
\end{enumerate}
\end{Thm}
\begin{proof}
\textbf{(\ref{en7})} Define the strong deformation retraction $H$ by setting $H_t(p,s):=(p,ts)$ for $(p,s)\in Y$ and $H_t(x):=x$ for $x\in X$. The claim follows by the representation~\eqref{e21}.\par
\textbf{(\ref{en8})} This is an immediate consequence of the fact that $X$ and $Y$ are geodesic and the representation \eqref{e21}.\par 
\textbf{(\ref{en9})} By the representation~\eqref{e21} the subset $X_{\gamma,L,R}\setminus X$ is locally isometric to $Y\setminus (\K \times \{0\})$. So
\begin{equation}
\label{e22}
\mathcal{H}^2(X_{\gamma,L,R})=\mathcal{H}^2(X)+\mathcal{H}^2(X_{\gamma,L,R}\setminus X)=\mathcal{H}^2(X)+\mathcal{H}^2(Y)=\mathcal{H}^2(X)+LR.
\end{equation}
\ \ \ \textbf{(\ref{en10})} For $p,q \in \K$ one has $d(\gamma(p),\gamma(q))=d_Y((p,0),(q,0))$.
So $l(\gamma_R)=L$ and $\gamma_R$ is $1$-chord-arc.\par 
\textbf{(\ref{en11})} Let $x \in X\subset X_{\gamma,L,R}$ and $r\in (0,R_0)$. Then for $y \in X_{\gamma,L,R}$ and $t\in [0,1]$ one has 
\begin{equation}
d(x,H_t(y))= d(H_t(x),H_t(y))\leq d(x,y).
\end{equation}
So for $r\in (0,R_0)$ the homotopy $H$ restricts to a strong deformation retraction of $B^d(x,r)$ onto $B^{d_X}(x,r)$ and by assumption $B^{d_X}(x,r)$ is contractible inside $B^{d_X}(x,\rho(r))\subset B^d(x,\rho(r))$.\par 
Now let $x \in X_{\gamma,L,R}$ arbitrary and $0< r <R_0'$. If $B^d(x,r)\cap X=\emptyset$, then $B^d(x,r)=B^{d_Y}(x,r)$ is homeomorphic to a ball in $\R^2$ and hence contractible inside itself. So assume $y \in X \cap B^d(x,r)$. Then $B^d(x,r)\subset B^d(y,2r)$. By our previous observation this is contractible inside $B^d(y,\rho(2r))$. But $B^d(y,\rho(2r))\subset B^d(x,\bar{\rho}(r))$. This proofs \ref{en11}.\par
\textbf{(\ref{en12})} Let $\eta:\K \rightarrow X_{\gamma,L,R}$ be a closed Lipschitz curve of length $l(\eta)<l_0$ and $D=\max\left\{(2\pi)\inv,L\inv R\right\}$.
\begin{Cla}
\label{cl1}
For every $D'>D$ there exists a Lipschitz homotopy $h:\eta\rightarrow \nu$ such that $l(\nu)\leq l(\eta)$, $\tn{Area}(h)\leq D'\cdot l(\eta)^2$ and either $\nu$ is constant or $\tn{im}(\nu)\subset X$.
\end{Cla} 
Assume the claim holds. By assumption there exists $v \in W^{1,2}(\Do,X_{\gamma,L,R})$ such that $\tn{tr}(v)=\nu$ and $\tn{Area}(v)\leq C \cdot l(\nu)^2$. So by lemma~\ref{l1} and claim~\ref{cl1} for every $D'>D$ there exists $u \in W^{1,2}(\Do,X_{\gamma,L,R})$ such that $
\tn{tr}(u)=\eta$ and
\begin{equation}
\label{e23}
\tn{Area}(u)=\tn{Area}(h)+\tn{Area}(u)\leq D'\cdot l(\eta)^2+C\cdot l(\nu)^2\leq (C+D')\cdot l(\eta)^2.
\end{equation}
So $X_{\gamma,L,R}$ satisfies a $(C'',l_0)$-quadratic isoperimetric inequality for every $C''>C'$. Hence by \cite{Crearb} $X_{\gamma,L,R}$ satisfies a $C'$-quadratic isoperimetric inequality.\end{proof}
\begin{proof}[Proof of claim~\ref{cl1}]
Let $U:=\eta\inv(X_{\gamma}\setminus X)$. Set $h^\eta_t(s):=H_t(\eta(s))$. Then one has $h^\eta:\eta \rightarrow \nu$ where $l(\nu)\leq l(\eta)$, $\tn{im}(\eta)\subset X$ and
\begin{equation}
\tn{Area}(\eta)\leq R \cdot l(\eta_{|U}).
\end{equation}
We consider three cases.\par
\textbf{\Rom{1}. $l(\eta)\geq L$:} In this case setting $h=h^\eta$ gives the desired homotopy as
\begin{equation}
\tn{Area}(h)\leq R \cdot l(\eta)\leq \frac{R}{L}\cdot l(\eta)^2.
\end{equation}
\ \ \ \textbf{\Rom{2}. $l(\eta)< L\land U =\K$:} As $l(\eta)<L$ one may fill $\eta$ as good as a planar curve. So there exists a Lipschitz homotopy $h:\eta\rightarrow \nu$ such that $\nu$ is constant and 
\[
\tn{Area}(h)\leq \frac{1}{4\pi}l(\eta)^2.
\]
\ \ \ \textbf{\Rom{3}. $l(\eta)< L\land U \neq \K$:} Taking $h=h^\eta$ would proof theorem~\ref{t12}.\ref{en12} upon replacing the constant $\frac{1}{2\pi}$ by $1$. For the finer result one has to homotope each piece of $\eta$ contained in $X_{\gamma}\setminus X$ to $X$ in a more optimal way.\par 
First note that by Reshetnyak's majorization theorem, see \cite{Res68}, for a planar Lipschitz curve $\alpha:[0,1]\rightarrow \R^2$ there is a compact convex  $C\subset \R^2$ and a $1$-Lipschitz map $m:C\rightarrow \R^2$ such that $m_{|\partial C}$ is length preserving and a parametrization of $\alpha\cdot \bar{\beta}$ where $\beta$ is a linear parametrization of the Euclidean segment $[\alpha(0),\alpha(1)]$. Hence there is a Lipschitz homotopy $g:\alpha\rightarrow \beta$ such that
\begin{equation}
\tn{Area}(g)\leq \Leb^2(C)\leq \Ha \cdot \frac{1}{4\pi} \cdot (2\cdot l(\alpha))^2=\frac{1}{2\pi}\cdot l(\alpha)^2.
\end{equation}
\ \par 
Decompose $U=\bigcupdot^\infty_{i=1}(v^i,w^i)$ where $(v^i, w^i)$ are possibly empty angular intervals. Set $\eta_1:=\eta$ and for $i\in \N$ let $k_i:=l(\eta_{|(v_i,w_i)})$. Using our observation on planar curves inductively for each $i\in \N$ there exists a Lipschitz homotopy $h^i:\eta_i\rightarrow \eta_{i+1}$ where $l(\eta_{i+1})\leq l(\eta_i)$, $\eta_{i+1}=\eta_i$ on $\K \setminus (v_i,w_i)$ and $\eta_{i+1}((v_i,w_i))\subset X$ such that
\begin{equation}
\tn{Area}(h^i)\leq \frac{1}{2\pi}\cdot k_i^2.
\end{equation}
Now let $\epsilon>0$ be arbitrary and $n\in \N$ such that 
$\sum^\infty_{i=n} k_i<\epsilon$. For $h=h_1\cdot h_2\cdot... \cdot h_{n-1} \cdot h^{\eta_{n}}:\eta\rightarrow \nu$ one has $l(\nu)\leq l(\eta)$, $\tn{im}(\nu)\subset X$ and
\begin{equation}
\tn{Area}(h)=\tn{Area}(h^{\eta_{n}})+\sum^{n-1}_{i=1} \tn{Area} (h^i)\leq \epsilon \cdot R+\sum^{n-1}_{i=1} \frac{1}{2\pi}\cdot k_i^2\leq \epsilon \cdot R+ \frac{1}{2\pi}\cdot l(\alpha)^2.
\end{equation}
As $\epsilon >0$ was arbitrary this proves the claim.
\end{proof}
\section{Limits of discs}
\label{4}
\subsection{Under chord-arc condition}
\label{4.1}
The following is a slightly weaker version of~$(6.4)$ in \cite{Why35}.
\begin{Thm}[\cite{Why35}]
\label{t13}
Let $\rho:(0,R)\rightarrow \R$ a contractibility function and $X$ a compact metric space. Let $Z_n \subset X$ be metric discs such that both $Z_n$ and $\partial Z_n$ are in $\tn{LGC}(\rho)$ when endowed with the subspace metric. If $Z_n\rightarrow_H Z$ and $\partial Z_n \rightarrow_H S$ where $Z,S\subset X$ are compact, then $Z$ is metric disc and $S=\partial Z$.
\end{Thm}
We will only use the following corollary of Whyburn's result.
\begin{Thm}
\label{t14}
Let $L\in (0,\infty)$, $\lambda\in [1,\infty)$ and $\rho:(0,R_0)\rightarrow \R$ a contractibility function. Let $(Z_n)$ be a sequence of metric discs such that $Z_n\in\tn{LGC}(\rho)$, $l(\partial Z_n)\leq L$ and $\partial Z_n$ is $\lambda$-chord-arc.\par
If $Z_n\rightarrow_{GH}Z$ for a compact metric space $Z$, then either $Z=\{*\}$ or $Z$ is a metric disc such that $l(\partial Z)\leq L$ and $\partial Z_n \rightarrow_{GH} \partial Z$.
\end{Thm}
\begin{proof}
If $l(\partial Z_n)\leq r<R_0$ then for $z\in Z_n$ one has $\partial Z_n\subset B(z,r)$ and there is a homotopy $h$ contracting $B(z,r)$ inside $B(z,\rho(r))$. By topological reasons the image of $h$ must be surjective onto $Z$ and hence $\tn{diam}(Z_n)\leq 2\rho(r)$. So if \[
\liminf_{n\rightarrow \infty} \ l(\partial Z_n)=0,\]
then $Z=\{*\}$.\par
So assume without loss of generality that $l(\partial Z_n)\geq l>0$ for all $n \in \N$. Then $\partial Z_n \in \tn{LGC}(\hat{\rho})$ where $\hat{R}_0:=\frac{l}{2\lambda}$, $\hat{\rho}:(0,\hat{R}_0)\rightarrow \R$ and $\hat{\rho}(r)=\lambda \cdot r$. By the proof of Gromov's compactnes criterion in \cite{Gro81a} there exists a compact metric space $X$ and isometric embeddings $Z_n\hookrightarrow X, Z\hookrightarrow X$ such that $Z_  n\rightarrow_H Z$ with respect to these embeddings. As $l(\partial Z_n)$ is uniformly bounded above by the Arzela-Ascoli theorem there exists a compact set $S\subset X$ such that $\partial Z_n \rightarrow_H S$. By theorem~\ref{t13} one has that $Z$ is a metric disc and $\partial Z=S$. Lower semicontinuity of length implies $l(\partial Z)\leq L$.
\end{proof}
Remember that $\mathcal{D}(L,A,C,l_0)$ is the class of geodesic metric discs $Z$ satisfying a $(C,l_0)$-quadratic isoperimetric inequality, $l(\partial Z)\leq L$ and $\mathcal{H}^2(Z)\leq A$. Let $\mathcal{D}^\lambda(L,A,C,l_0)$ be the set of those $Z\in \mathcal{D}(L,A,C,l_0)$ such that $\partial Z$ is $\lambda$-chord-arc.
\begin{Thm}
\label{t15}
$\mathcal{D}^\lambda(L,A,C,l_0)\cup \{*\}$ is compact.
\end{Thm}
\begin{proof}
By corollary~\ref{c1} $\mathcal{D}(L,A,C,l_0)$ is precompact and hence so is $\mathcal{D}^\lambda(L,A,C,l_0)$. So it remains to show that $\mathcal{D}^\lambda(L,A,C,l_0)\cup \{*\}$ is closed.\par 
Let $Z^n \in \mathcal{D}^\lambda(L,A,C,l_0)$ and $Z_n \rightarrow_{GH} Z$. By theorem~\ref{t10}.\ref{en5} there is a contractibility function $\rho$ such that $Z^n\in \tn{LGC}(\rho)$. So theorem~\ref{t14} implies that either $Z=\{*\}$ or $Z$ is a metric disc such that $l(\partial Z) \leq L$ and $\partial Z_n \rightarrow_{GH} \partial Z$. Assume we are in the later case.\par
By lower semicontinuouity of length one has that $\partial Z$ is $\lambda$-chord-arc. The conditions of being geodesic and satisfying a $(C,l_0)$-quadratic isoperimetric inequality are both stable under Gromov-Hausdorff limits, see \cite{BBI01} and \cite{Crearb} respectively. Note that the proof of theorem~\ref{t3} is completed at this point. For the local version it only remains to check $\mathcal{H}^2(Z)\leq A$.\par 
As in the proof of theorem~\ref{t9} all $Z_n$ satisfy a $\hat{C}$-quadratic isoperimetric inequality where $\hat{C}=\hat{C}(A,C,l_0)$. So by theorem~\ref{t7}, theorem~\ref{t8} and theorem~\ref{t9} there exist constants $M\in[0,\infty)$, $p\in (2,\infty)$ depending only on $\hat{C},\lambda$ and maps \[
u_n \in W^{1,p}(\Do,Z_n)\cap C^0(\Dc,Z_n)\cap \Lambda(\partial Z_n, Z_n)
\]
such that $\tn{Area}(u_n)=\mathcal{H}^2(Z_n)$ and
\begin{equation}
\label{e31}
\left(E^p_+(u_n)\right)^{\frac{1}{p}}\leq M \cdot \left(E^2_+(u_n)\right)^{\frac{1}{2}}\leq M\sqrt{2\cdot\tn{Area}(u_n)}=M\sqrt{2\cdot\mathcal{H}^2(Z_n)}\leq M\cdot \sqrt{2A}
\end{equation}
for all $n\in \N$. Especially by theorem~\ref{t4} the maps $u_n$ are uniformly Hölder and hence equicontinuous.\par 
As in the proof of theorem~\ref{t14} we may assume that $Z_n$ and $Z$ are subsets of a compact metric space $X$ and $Z_n \rightarrow_H Z$. Then by Arzela-Ascoli $u_n:\Dc\rightarrow X$ converges uniformly to a map $u\in C^0(\Dc,X)$. Especially $u_n \rightarrow_{L^p} u$ and hence by~\ref{t6} and \eqref{e31} $u \in W^{1,p}(\Do,X)$. So by theorem~\ref{t4} $u$ satisfies property Lusin's property~(N). As $u_n:\Dc\rightarrow Z_n$ is surjective one has $\tn{im}(u)=Z$. So by theorem~\ref{t5} and theorem~\ref{t6}
\begin{equation}
\label{e32}
\mathcal{H}^2(Z)\leq \tn{Area}(u)\leq \liminf_{n\rightarrow \infty} \tn{Area}(u_n)\leq A.
\end{equation}
and hence $Z \in \mathcal{D}^\lambda(L,A,C,l_0)$.
\end{proof}
\subsection{Without chord-arc condition}
\label{4.2}
Recall that a metric space $Z$ is called a disc retract if there is a Lipschitz curve $\gamma:\K \rightarrow X$ such that $X_\gamma$ is a topological disc. We say that $l(\partial Z)\leq L$ if there exists such $\gamma$ satisfying $l(\gamma)\leq L$. The following lemma shows that one may always assume $\gamma$ to be parametrized by constant speed.
\begin{Lem}
\label{l3}
Let $Z$ be a compact metric space, $\gamma:\K \rightarrow Z$ be a Lipschitz curve and $\bar{\gamma}:\K \rightarrow Z$ the constant speed parametrization of $\gamma$. If $Z_\gamma$ is a topological disc, then $Z_{\bar{\gamma}}$ is a topological disc.
\end{Lem}
\begin{proof}
Let $M:\K \rightarrow \K$ be monotone and such that $\gamma=\bar{\gamma}\circ M$. Define $P:X_{\gamma}\rightarrow X_{\bar{\gamma}}$ by \begin{equation}
\label{e33}
P(x)=\begin{cases}
(M(p),t) &; x=(p,t)\in \K \times [0,1]\\
x &;x \in X
\end{cases}
\end{equation}
Then $P$ is continuous. Furthermore if $x \in X$, then $P^{-1}(x)=\{x\}$ and if $p\in \K,t \in (0,1]$, then $P^{-1}((p,t))=I\times \{t\}$ where $I \subset \K$ is a closed interval. So $P$ defines a cell-like map $\Dc \cong X_{\gamma}\rightarrow X_{\bar{\gamma}}$ such that $P_{|\partial X_{\gamma}}$ is cell-like too. So by corollary~$7.12$ in \cite{LW18a} $X_{\bar{\gamma}}$ is homeomorphic to $\Dc$.
\end{proof}
Remember that $\mathcal{E}(L,A,C,l_0)$ is the set of geodesic metric disc retracts satisfying a $(C,l_0)$-quadratic isoperimetric inequality, $l(\partial Z)\leq L$ and $\mathcal{H}^2(Z)\leq A$.
\begin{Thm}
\label{t16}
Let $L,A,C\in (0,\infty)$ and $l_0 \in (0,\infty]$. Then $\mathcal{E}(L,A,C,l_0)$ is compact and
\begin{equation}
\mathcal{D}(L,A,C,l_0)\subset \mathcal{E}(L,A,C,l_0)\subset \overline{\mathcal{D}\left(L,A',C+(2\pi)\inv,l_0\right)}
\end{equation}
for every $A'>A$.
\end{Thm}
\begin{proof}
It is clear that $\mathcal{D}_{L,A,C,l_0}\subset \mathcal{E}(L,A,C,l_0)$.\par 
Let $Z \in \mathcal{E}(L,A,C,l_0)$. By lemma~\ref{l3} there is a constant-speed curve $\gamma:\K \rightarrow X$ such that $l(\gamma)\leq L$ and $Z_\gamma$ is a disc. Let $Z^n:=Z_{\gamma,L,\frac{1}{n}}$. By theorem~\ref{t12} for $n>L\inv $ one has that $d_{GH}(Z,Z^n)\leq \frac{1}{n}$ and 
\begin{equation}
\label{e34}
Z^{n} \in \mathcal{D}\left(L,A+\frac{L}{n},C+\frac{1}{2\pi},l_0\right).
\end{equation}
So $\mathcal{E}_{L,A,C,l_0}\subset \overline{\mathcal{D}\left(L,A',C+(2\pi)\inv,l_0\right)}$.\par 
By corollary~\ref{c1} $\mathcal{D}\left(L,A',C+(2\pi)\inv,l_0\right)$ is precompact. So the only thing that remains to be checked is that $\mathcal{E}(L,A,C,l_0)$ is closed.\par
Let $Z^n\in \mathcal{E}(L,A,C,l_0)$ and $Z^n\rightarrow_{GH} Z$ where $Z$ is a compact metric space. Then as in the proof of theorem~\ref{t15} $Z$ is geodesic and satisfies a $(C,l_0)$-quadratic isoperimetric inequality. Let $\gamma^n:\K \rightarrow X$ be constant speed curves such that $l(\gamma^n) \leq L$ and $Z^n_{\gamma^n}$ is a disc. Again we may assume that all $Z_n$ and $Z$ are subsets of a compact metric space $X$ such that $Z^n\rightarrow_{H} Z$. As $\gamma^n$ are uniformly Lipschitz by the Arzela-Ascoli theorem up to passing to a subsequence there exists a limit curve~$\gamma:\K \rightarrow Z$ and $l(\gamma)\leq L$. \par
Set $W^n:=Z^n_{\gamma^n,L,L}$. Then by theorem~\ref{t12} one has
\begin{equation}
W^n \in \mathcal{D}^1(L,A+L^2,C+1,l_0)
\end{equation}
and hence by theorem~\ref{t15} after passing to a subsequence there is 
\[
W\in \mathcal{D}^1(L,A+L^2,C+1,l_0)
\] such that $W^n\rightarrow_{GH} W$. But by the explicit formula~\eqref{e21} in this case $W$ has to be isometric to $Z_{\gamma,L,L}$. So $Z$ is a disc retract and $l(\partial Z)\leq L$. Furthermore by theorem~\ref{t12}.\ref{en9}.
\begin{equation}
\label{e35}
\mathcal{H}^2(Z)=\mathcal{H}^2(W)-L^2\leq A+L^2-L^2=A.
\end{equation}
So $Z \in \mathcal{E}(L,A,C,l_0)$.
\end{proof}
\begin{Thm}
\label{t17}
Let $\rho$ be a contractibility function, $L\geq 0$ and $Z$ a compact metric space. Let $Z_n$ be a sequence of disc retracts such that $l(\partial Z_n)\leq L$ for all $n$ and $Z_n\rightarrow_{GH} Z$. Then $Z$ is a disc retract and $l(\partial Z)\leq L$. 
\end{Thm}
\begin{proof}
Proceed exactly as in the last part of the proof of theorem~\ref{t16} applying theorem~\ref{t12}.\ref{en11} instead of theorem~\ref{t12}.\ref{en12} and theorem~\ref{t14} instead of theorem~\ref{t15}.
\end{proof}
\section{Acknowledgements}
\label{5}
I would like to thank my PhD advisor Alexander Lytchak for great support in everything.
\bibliographystyle{alpha}
\bibliography{BibClosureMinDiscs}
\end{document}